\documentclass[reqno]{article}
\usepackage{amssymb}
\usepackage{amsthm}
\usepackage{amsmath, enumerate}
\usepackage{xcolor}

\def\added{\relax}

 \makeatletter
 \renewcommand\section{\@startsection {section}{1}{\z@}%
 {-3.5ex \@plus -1ex \@minus -.2ex}%
 {2.3ex \@plus.2ex}%
 {\center \normalfont\large\bfseries}}
 \makeatother

\newtheorem{thm}{Theorem}

\newtheorem{lem}[thm]{Lemma}
\newtheorem{defi}[thm]{Definition}
\newtheorem{remark}[thm]{Remark}
\newtheorem{example}[thm]{Example}
\newtheorem{pb}[thm]{Problem}

\newenvironment{rk}{\begin{remark}\rm}{\end{remark}}

\newcommand{\real}{{\mathbb R}}
\newcommand{\nat}{{\mathbb N}}

\newcommand{\com}{{\mathbb C}}

\newcommand{\M}{{\mathcal M}}

\newcommand{\Rad}{{{\rm Rad}}}

\renewcommand{\a}{\alpha}
\renewcommand{\b}{\beta}
\newcommand{\g}{\gamma}

\newcommand{\e}{\varepsilon}

\renewcommand{\l}{\lambda}

\renewcommand{\e}{\varepsilon}

\renewcommand{\O}{{\Omega}}
\renewcommand{\o}{{\omega}}
\newcommand{\s}{\sigma}
\renewcommand{\i}{{\rm i}}

\newcommand{\8}{\infty}
\newcommand{\el}{\ell}

\newcommand{\la}{\langle}
\newcommand{\ra}{\rangle}
\newcommand{\wt}{\widetilde}
\newcommand{\wh}{\widehat}
\newcommand{\n}{\noindent}
\newcommand{\pf}{\noindent{\it Proof.~~}}
\newcommand{\cqd}{\hfill$\Box$}
\newcommand{\be}{\begin{eqnarray*}}
\newcommand{\ee}{\end{eqnarray*}}
\newcommand{\beq}{\begin{equation}}
\newcommand{\eeq}{\end{equation}}

\begin{document}

\title{On the vector-valued  Littlewood-Paley-Rubio de Francia inequality}

\author{Denis Potapov\thanks{School of Mathematics and Statistics,
    University of NSW, Kensington NSW 2052, Australia (first two
    authors); E-mail: \texttt{d.potapov@unsw.edu.au}} \and Fedor
  Sukochev\thanks{E-mail: \texttt{f.sukochev@unsw.edu.au}} \and
  Quanhua Xu\thanks{School of Mathematics and Statistics, Wuhan
    University, Wuhan 430072, China and Laboratoire de
    Math\'ematiques, Universit\'e de Franche-Comt\'e, 25030 Besan\c
    con cedex, France; E-mail: \texttt{qxu@univ-fcomte.fr} }}

\date{}

\maketitle

\begin{abstract}
  {\added The paper studies Banach spaces satisfying the
    Littlewood-Paley-Rubio de Francia property~$LPR_p$, $2 \leq p <
    \infty$.  The paper shows that every Banach lattice whose
    $2$-concavification is a UMD Banach lattice has this property.
    The paper also shows that every space having~$LPR_q$ also
    has~$LPR_p$ with~$q \leq p < \infty$.

  
    }
\end{abstract}


\newcommand{\In}{{\mathcal{I}}}

\bibliographystyle{short}

\section{Introduction}

Let~$X$ be a Banach space and~$L^p(\real; X)$ be the Bochner space of
$p$-integrable $X$-valued functions on $\real$. If~$X = \com$, we
abbreviate~$L^p(\real; X) = L^p(\real)$, $1 \leq p < \infty$. For
every~$f \in L^1(\real; X)$, $\hat f$ stands for the Fourier
transform.  If~$I \subseteq \real$ is an interval, then~$S_I$ is the
Riesz projection adjusted to the interval~$I$, i.e., $$ S_I f(t) =
\int_I \hat f(s)\, e^{2\pi{\rm i}st}\, ds. $$ The following remarkable
inequality was proved by J.L.~Rubio de Francia in~\cite{R1985}.  For
every~$2 \leq p < \infty$, there is a constant~$c_p$ such that for
every collection of pairwise disjoint intervals~$\left( I_j \right)_{j
  = 1}^\infty$, the following estimate holds
\begin{equation}
  \label{RubioEstimate}
  \left\| \left( \sum_{j = 1}^\infty \left|  S_{I_j}
        f \right|^2 \right)^{\frac 12} \right\|_{L^p (\real)}
  \leq c_p\, \left\| f \right\|_{L^p (\real)} \,, \quad\forall\; f\in
  L^p(\real). 
\end{equation}

In this note, we shall discuss the version of the theorem above when
functions take values in a Banach space $X$.  Let~$\left(
  \varepsilon_k \right)_{k \geq 1}$ be the system of Rademacher
functions on~$[0, \,1]$.  {\added The space~$\Rad (X)$ is the closure
  in~$L^p([0, 1]; X)$, $1 \leq p < \infty$ of all $X$-valued functions
  of the form $$ g(\omega) = \sum_{k = 1}^n \varepsilon_k (\omega)\,
  x_k,\ \ x_k \in X,\ \ n \geq 1. $$ The above definition is
  independent of~$1 \leq p < \infty$.  It follows from the
  Khintchine-Kahane inequality (see~\cite{LT-II}).  In fact, the above
  fact is a consequence of a, so-called, {\it contraction
    principle\/}.  It states that, for every sequence of
  elements~$\left\{x_j \right\}_{j=1}^\infty \subseteq X$ and sequence
  of complex numbers~$\left\{\alpha_j \right\}_{j = 1}^\infty$ such
  that~$\left| \alpha_j \right| \leq 1$, $j \geq 1$, the following
  inequality holds $$ \left\| \sum_{j = 1}^\infty \alpha_j\,
    \epsilon_j\, x_j \right\|_{L^p(\real, \Rad(X))} \leq \, c_p \,
  \left\| \sum_{j =1}^\infty \epsilon_j\, x_j \right\|_{L^p(\real,
    \Rad(X))}. $$ We shall employ this principle on numerous occasions
  in this paper.}

Following~\cite{BGT2003}, we shall call~$X$ {\it a space with the
  $\hbox{LPR}_p$ property} with $2\le p<\8$, if there exists a
constant $c>0$ such that for any collection of pairwise disjoint
intervals~$\left\{I_j \right\}_{j =1}^\infty$ we have that
 \begin{equation}
  \label{LPRpEst}
  \left\| \sum_{j = 1}^\infty \varepsilon_j S_{I_j} f
  \right\|_{L^p(\real; \Rad (X))} \leq c\, \left\| f \right\|_{L^p(\real;
    X)} \,, \quad\forall\; f\in
 L^p(\real;X). 
 \end{equation}
 It was proved in~\cite{HTY2009} that every space with the
 $\hbox{LPR}_p$ property is necessarily UMD and of type~$2$. It is an
 open problem whether the converse is true. It is also unknown whether
 $\hbox{LPR}_p$ is independent of $p$. Note that Rubio de Francia's
 inequality says that $\com$ has the $\hbox{LPR}_p$ property for every
 $2\le p<\8$. By the Khintchine inequality and the Fubini theorem we
 see that any $L^p$-space with~$2 \leq p < \infty$ has the
 $\hbox{LPR}_p$ property. Using interpolation, we deduce that a
 Lorentz space $L^{p, r}$ has the $\hbox{LPR}_q$ property for some
 indices $p, r$ and $q$.  However, until recently there were no
 non-trivial examples of spaces with~$\hbox{LPR}_p$ found.
 
 If $X$ is a Banach lattice, estimate \eqref{LPRpEst} admits a
 pleasant form as in the scalar case:
 \begin{equation}
  \label{LPRpEstlat}
  \left\| \left( \sum_{j = 1}^\infty \left|  S_{I_j}
        f \right|^2 \right)^{\frac 12} \right\|_{L^p (\real; X)}
  \leq c\, \left\| f \right\|_{L^p (\real; X)} \,, \quad\forall\; f\in
 L^p(\real;X). 
 \end{equation}
 We shall show that if the $2$-concavification $X_{(2)}$ of~$X$ is a
 UMD Banach lattice, then \eqref{LPRpEstlat} holds for all $2<p<\8$,
 so $X$ is a space with the $\hbox{LPR}_p$ property. Recall that
 $X_{(2)}$ is the lattice defined by the following quasi-norm
 $$ \left\| f\right\|_{X_{(2)}} = \left\| \,
  \left| f \right|^{\frac 12} \right\|^2_{X}. $$ 
The space~$X_{(2)}$
is a Banach lattice if and only if $X$ is $2$-convex, i.e., 
 $$ \left\| \left(
    \sum_{j = 1}^n \left| f_j \right|^2 \right)^{\frac 12} \right\|_X
 \leq \left( \sum_{j = 1}^n \left\| f_j \right\|^2_X \right)^{\frac
  12}. $$
We refer to \cite{LT-II} for more information on Banach lattices.

We shall also show that if~$X$ is a Banach
space (not necessarily a lattice) with the $\hbox{LPR}_q$ property for some $q$, then $X$ has   the $\hbox{LPR}_p$ property for
every~$p \geq q$.

\section{Dyadic decomposition}
\label{sec:dyadic-decomposition}

For every interval~$I \subseteq \real$, let~$2I$ be the interval of
double length and the same centre as~$I$.  Let~$\In = \left( I_j
\right)_{j = 1}^\infty$ be a collection of pairwise disjoint
intervals.  We set~$2\In = \left( 2 I_j \right)_{j = 1}^\infty$.  The
collection~$\In$ is called {\it well-distributed\/} if 
there is a number~$d$ such that each element of~$2\In$ intersects at
most~$d$ other elements of~$2\In$.

In this section, we fix a pairwise disjoint collection of
intervals~$\left( I_j \right)_{j=1}^\infty$ and we break each
interval~$I_j$, $j \geq 1$ into a number of smaller dyadic
subintervals such that the new collection is well-distributed.  This
construction was employed in a number of earlier papers.

We start with two elementary remarks on estimate \eqref{LPRpEst} or
\eqref{LPRpEstlat}. Firstly, it suffices to consider a finite sequence
$(I_j)_j$ of disjoint finite intervals. Secondly, by dilation, we may
assume $|I_j|\geq 4$ for all $j$.  Thus all sums on $j$ and $k$ in
what follows are finite. Fix~$j \geq 1$.  Let $I_j=(a_j, b_j]$. Let
$n_j = \max\{n\in\nat:\; 2^{n+1}\leq b_j-a_j+4\} $. We first split
$I_j$ into two subintervals with equal length
 \[
 I_j^a=(a_j,\;\frac{a_j+b_j}2]\quad\mbox{and}
  \quad I_j^b=(\frac{a_j+b_j}2,\;b_j].
 \]
 Then we decompose $I_j^a$ and $I_j^b$ into relative dyadic
 subintervals as follows: \be
 I_j^a=\bigcup_{k=1}^{n_j}(a_{j,k},\;a_{j,k+1}]\quad\mbox{and}\quad
 I_j^b=\bigcup_{k=1}^{n_j}(b_{j,k+1},\; b_{j,k}], \ee where \be
 a_{j,k}&=&a_j-2+2^k,\quad 1\leq k\leq n_j \quad \mbox{and}\quad
 a_{j,n_j+1}=\frac{a_j+b_j}2; \\
 b_{j,k}&=&b_j+2-2^k, \quad 1\leq k\leq n_j \quad\mbox{and}\quad
 b_{j,n_j+1}=\frac{a_j+b_j}2.  \ee Let
 \[
 I_{j, k}^a=(a_{j,k},\;a_{j,k+1}],\quad I_{j, k}^b=(b_{j,k+1},\;b_{j,k}]
\]
for $1\leq k\leq n_j$ and let $I_{j, k}^a,\, I_{j, k}^b$ be the empty
set for the other $k $'s. Also put
 \[
 \tilde I^a_{j, n_j}=(a_j-2+2^{n_j},\;
 a_j-2+2^{n_j+1}]\quad\mbox{and}\quad
 \tilde I^b_{j, n_j}=(b_j+2-2^{n_j+1},\;b_j+2-2^{n_j}].
 \]

 \begin{lem}
   \label{ToDydic}
   A Banach space~$X$ has the $\hbox{LPR}_p$ property if there is a
   constant~$c > 0$ such that
   \begin{equation}
     \label{a}
     \max_{u = a, b} \left\| \sum_{j = 1}^\infty \e_j \sum_{k = 1}^{n_j}
       \e'_k S_{I_{j, k}^u}f \right\|_{L^p(\real; \Rad_2 (X))} \leq c \, 
     \|f\|_{L^p(\real; X)} \,, \quad\forall\; f\in
 L^p(\real;X),
   \end{equation}
   where~$\Rad_2( X) = \Rad (\Rad' (X))$ and~$\Rad' (X)$ is the space
   with respect to another copy of the Rademacher system~$\left(
     \varepsilon'_k \right)_{k \geq 1}$.
 \end{lem}

 {\added Observe that if~\eqref{a} holds, for every family of
   intervals~$\left( I_j \right)_{j = 1}^\infty$, then $X$ is a UMD
   space.  Indeed,~(\ref{a}) implies that $$ \left\| S_{I^u_{j, k}}
     f\right\|_{L^p(\real, X)} \leq c\, \left\| f\right\|_{L^p(\real,
     X)},\ \ u = a, b,\ j\geq 1,\ 1 \leq k \leq n_j. $$ That is, by
   adjusting the choice of intervals, it implies that every
   projection~$S_I$ is bounded on~$L^p(\real, X)$ and $$ \sup_{I
     \subseteq \real} \left\| S_I \right\|_{L^p(\real, X) \mapsto
     L^p(\real, X)} < + \infty. $$ The latter is equivalent to the
   fact that~$X$ is UMD (see~\cite{Bo1986})}.

\begin{proof}
Let $f\in L^p(\real; X)$.
Then
\begin{equation*}
  \left\|\sum_{j =1 }^\infty \e_jS_{I_j} f\right\|_{L^p(\real; \Rad (X))}
  \le \left\|\sum_{j = 1}^\infty \e_j S_{I_j^a}f\right\|_{L^p(\real; \Rad (X))} +
  \left\|\sum_{j = 1}^\infty \e_jS_{I_j^b}f\right\|_{L^p(\real; \Rad (X))}\,.
\end{equation*}
Using the subintervals $I_{j, k}^a$ and {\added the contraction
  principle}, we write
\begin{eqnarray*}
  \left\|\sum_{j = 1}^\infty \e_j S_{I_j^a}f\right\|_{L^p(\real; \Rad (X))}
  & = & \left\|\sum_{j =1}^\infty \sum_{k = 1}^{n_j}
    \e_j S_{I_{j, k}^a}f\right\|_{L^p(\real; \Rad (X))} \\  
  & \sim & \left\|\sum_{j = 1}^\infty \sum_{k = 1}^{n_j} \e_j\exp(-2\pi\i
    a_j\,\cdot)S_{I_{j, k}^a}f\right\|_{L^p(\real; \Rad (X))}. 
\end{eqnarray*}
Note that
 \[
 \exp(-2\pi\i  a_j\,\cdot)S_{I_{j, k}^a}f =S_{I_{j, k}^a-a_j}[\exp(-2\pi\i
 a_j\,\cdot)f]
 \]
and
 \[
 I_{j, k}^a - a_j=(2^k-2,\; 2^{k+1}-2],\ k<n_j; \quad
 I_{j, n_j}^a-a_j\subseteq (2^{n_j}-2,\; 2^{n_j+1}-2].
 \]
 {\added Recall that $X$ is a UMD space.}  Therefore, applying
 Bourgain's Fourier multiplier theorem (see~\cite{Bo1986}) to the
 function
 \[
 \sum_{j = 1}^\infty \sum_{k = 1}^{n_j} \e_j\exp(-2\pi\i
 a_j\,\cdot)S_{I_{j, k}^a}f \in L^p(\real;\Rad (X))),
 \]
 we obtain {\added (the contraction principle being used in the last
   step)}
\begin{multline*}
  \left\| \sum_{j = 1}^\infty \sum_{k = 1}^{n_j} \e_j \exp(-2\pi\i
    a_j\,\cdot)S_{I_{j, k}^a} f \right\|_{L^p(\real; \Rad (X))}
  \sim \\
  \left\|\sum_{j = 1}^\infty \sum_{k = 1}^{n_j} \e_j\e_k'
    \exp(-2\pi\i a_j\,\cdot) S_{I_{j, k}^a} f \right\|_{L^p(\real; \Rad_2( X))}
  \\ \sim \left\| \sum_{j = 1}^\infty \sum_{k = 1}^{n_j} \e_j\e_k'
    S_{I_{j, k}^a} f \right\|_{L^p(\real; \Rad_2( X))}.
\end{multline*}
Similarly,
\[
\left\|\sum_{j = 1}^{\infty} \e_j S_{I_j^b}f\right\|_{L^p(\real; \Rad
  X)} \sim \left\|\sum_{j = 1}^{\infty} \sum_{k = 1}^{n_j}
  \e_j\e'_kS_{I_{j, k}^b}f\right\|_{L^p(\real; \Rad_2( X))}\,.
 \]
 Combining the preceding estimates, we get
 \begin{multline*}
   \left\|\sum_{j = 1}^{\infty} \e_j S_{I_j}f\right\|_{L^p(\real; \Rad
     X)} \le c_p \, \left[ \left\|\sum_{j = 1}^{\infty} \sum_{k =
         1}^{n_j} \e_j\e'_kS_{I_{j, k}^a}f\right\|_{L^p(\real; \Rad_2(
       X))}  \right. \\ \left. + \left\|\sum_{j = 1}^{\infty} \sum_{k = 1}^{n_j}
       \e_j\e'_kS_{I_{j, k}^b}f\right\|_{L^p(\real; \Rad_2( X))}
   \right].
 \end{multline*}
\end{proof}

{\added Let us observe that, if $X$ is a UMD space, the argument in
  the proof above shows that $$ \left\|\sum_{j = 1}^{\infty} \e_j
    S_{I_j}f\right\|_{L^p(\real; \Rad X)}\lesssim \max_{u = a, b}
  \left\| \sum_{j = 1}^\infty \e_j \sum_{k = 1}^{n_j} \e'_k S_{I_{j,
        k}^u}f \right\|_{L^p(\real; \Rad_2 (X))} \,. $$ Moreover, the
  argument can be reversed to show the opposite estimate (see the
  proof of~(\ref{RemarkIIi-op}) below.) This observation is summarised
  in the following remark.}

\begin{rk}
  \label{DyadicLattice}
  \begin{enumerate}[i)]
  
  \item If $X$ is a UMD space, then $$ \left\|\sum_{j = 1}^{\infty}
      \e_j S_{I_j}f\right\|_{L^p(\real; \Rad X)}\lesssim \max_{u = a,
      b} \left\| \sum_{j = 1}^\infty \e_j \sum_{k = 1}^{n_j} \e'_k
      S_{I_{j, k}^u}f \right\|_{L^p(\real; \Rad_2 (X))} \,. $$
  
  \item If\/~$\In = \left( I_j \right)_{j \geq 1}$ is a collection of
    pairwise disjoint intervals and~$\In_u = \left( I^u_{j, k}
    \right)_{j \geq 1, 1 \leq k \leq n_j}$, $u = a, b$, then both
    collections~$\In_a$ and~$\In_b$ are well-distributed.

  \item {\added If~$X$ is a Banach lattice then it has the
      $\alpha$-property (see~\cite{PisierAlph}).  That is,} $$ \left\|
      \sum_{j, k = 1}^\infty \varepsilon_j \varepsilon'_k x_{jk}
    \right\|_{\Rad_2( X)} \sim \left\| \sum_{j, k = 1}^\infty
      \varepsilon_{jk} x_{jk} \right\|_{\Rad (X)}\,,$$ where
    $(\varepsilon_{jk})$ is an independent family of Rademacher
    functions.
    
  \item The above two observations imply that if~$X$ is a Banach
    lattice, then it has the $\hbox{LPR}_p$ property if and only if 
    estimate~(\ref{LPRpEst}) holds for every well-distributed
    collection of intervals~$\In$.
  \end{enumerate}
 \end{rk}

\section{LPR-estimate for Banach lattices}
\label{sec:LPR-lattices}

\begin{thm}
  \label{RubioForLattices}
  If $X$ is a Banach lattice such that~$X_{(2)}$ is a UMD {\added
    Banach} space, then~$X$ has the $\hbox{LPR}_p$ property for
  every~$2 < p < \infty$.
\end{thm}

We shall need the following remark for the proof.

\begin{rk}
  \label{RboundedRem}
  If~$X$ is UMD and~$1 < p < \infty$, then the
  family~$\left\{S_{I}\right\}_{I \subseteq {\mathcal I}}$ is
  $R$-bounded (see~\cite{CPSW2000}), i.e., $$ \left\| \sum_{I
      \subseteq {\mathcal{I}}} \epsilon_I S_I f_I \right\|_{L^p(\real;
    \Rad(X))} \leq c_X\, \left\| \sum_{I \subseteq {\mathcal{I}}}
    \epsilon_I f_I \right\|_{L^p(\real; \Rad(X))}. $$
\end{rk}

\begin{proof}[Proof of Theorem~\ref{RubioForLattices}]
  The proof directly employs the pointwise estimate of~\cite{R1985}.
  We assume, that $X$ is a K\"othe function space on a measure space
  $(\O, \mu)$.
  
  Let~$f \in L^1_{\hbox{\scriptsize loc}}(\real; X)$.  Let $M(f)$ be
  the Hardy-Littlewood maximal function of~$f$, i.e., 
   $$ M (f)(t) =\sup_{I \subseteq \real \atop t \in I} \frac 1{\left| I \right|}
   \int_I \left| f(s) \right|\, ds $$ and $$ M_2(f) = \left[ M \left|
       f \right|^2 \right]^{\frac 12}. $$ Let $$ f^\sharp (t) =
   \sup_{I \subseteq \real \atop t \in I} \frac {1}{\left| I \right|}
   \int_I \left| f(s) - f_I \right|\, ds,\ \ f_I = \frac 1 {\left| I
     \right|} \int_I f(s)\, ds. $$ Note that $M (f)$ is a function of
   two variables $(t, \omega)$: for each fixed $\omega$, $M (f)(\cdot,
   \omega)$ is the usual Hardy-Littlewood maximal function of
   $f(\cdot,\omega)$. The same remark applies to $M_2(f)$ and
   $f^\sharp$. For $f$ sufficiently nice (which will be assumed in the
   sequel), all these functions are well-defined.

  Observe that due to Remark~\ref{DyadicLattice} we have only to show
  estimate~(\ref{LPRpEst}) for a well-distributed family of
  intervals. Let us fix a family of pairwise disjoint intervals~$\In$
  and let us assume that~$\In$ is well-distributed.  Fix a Schwartz
  function~$\psi(t)$ whose Fourier transform satisfies $$ \chi_{[-1/2,
    1/2]} \leq \hat \psi \leq \chi_{[-1, 1]}. $$ If~$I \in \In$, then
  we set
  \begin{equation*}
    \psi_I(t) = \left| I \right| \exp(2\pi\i c_I t)\,\psi(\left| I
    \right| t),
  \end{equation*}
  where~$c_I$ is the centre of~$I$.  The Fourier transform of $\psi_I$
  is adapted to $I$, i.e. 
  \begin{equation*}
    \chi_I \leq \hat \psi_I\leq \chi_{2I}.
  \end{equation*}
  {\added In particular, $$ S_I (f) = \psi_I * S_I(f). $$ Consequently,
    from the Khintchine inequality and Remark~\ref{RboundedRem}, $$
    \left\| \left( \sum_{I \in \In} \left| S_I (f) \right|^2
      \right)^{\frac 12} \right\|_{L^p(\real, X)} \leq c_p\, \left\|
      G(f) \right\|_{L^p(\real, X)},\ \ 1 < p < \infty, $$ where } $$
  G(f) = \left( \sum_{I \in \In} \left| \psi_I * f \right|^2
  \right)^{\frac 12},\ \ f \in L^1(\real; X). $${\added Thus, to
    finish the proof, we need to show that $$ \left\| G(f)
    \right\|_{L^p(\real, X)} \leq c_p\, \left\| f\right\|_{L^p(\real,
      X)},\ \ 2 < p < \infty. $$} It was shown in~\cite{R1985}
  that~$G(f (\cdot, \omega))^\sharp$ is almost everywhere dominated by
  $M_2(f (\cdot, \omega))$, i.e., $$ G(f (\cdot, \omega))^\sharp \le
  c\, M_2(f (\cdot, \omega)),\ \ \text{a.e.}\ \omega \in \O, $$ for
  some universal $c > 0$.  Since $$ G(f)(t, \o)=G(f(\cdot,
  \o))(t)\quad\mbox{and}\quad M_2(f)(t, \o)=M_2(f(\cdot, \o))(t),\quad
  t\in \real, \; \o\in\O, $$ we clearly have that $$ G(f)^\sharp \leq
  c\, M_2(f). $$ Therefore,
 $$\big\|G(f)^\sharp\big\|_{L^p(\real;X)}\le
 c \big\|M_2(f)\big\|_{L^p(\real;X)}\,.$$ It remains to prove
 $$\big\|G(f)\big\|_{L^p(\real;X)}\le
 C\big\|G(f)^\sharp\big\|_{L^p(\real;X)}\quad\mbox{and}\quad
 \big\|M_2(f)\big\|_{L^p(\real;X)}\le
 C\big\|f\big\|_{L^p(\real;X)}\,.$$ The second inequality above
 immediately follows from Bourgain's maximal inequality for UMD
 lattices (applied to $X_{(2)}$ here,
 see~\cite[Theorem~3]{Rubio1986}):
 $$\big\|M_2(f)\big\|_{L^p(\real;X)}^2
 =\big\|M(|f|^2)\big\|_{L^{\frac{p}2}(\real;X_{(2)})}
 \le C \big\||f|^2\big\|_{L^{\frac{p}2}(\real;X_{(2)})}
 =C \big\|f\big\|_{L^p(\real;X)}^2\,.$$
It remains to show the first one. To this end we shall prove the
following inequality (for a general $f$ instead of $G(f)$)
 $$\big\|f\big\|_{L^p(\real;X)}\le
 C\big\|f^\sharp\big\|_{L^p(\real;X)}\,.$$
This is again an immediate consequence of the following classical
duality inequality (see~\cite[p.~146]{Stein-HA})
 $$\left|\int_{\real} u v\right|\le C\int_{\real} u^\sharp \M(v)$$
 for any $u\in L^p(\real)$ and $v\in L^{p'}(\real)$, where $\M(v)$
 denotes the grand maximal function of $v$.  Note that $\M(v)\le
 CM(v)$. Now let $g\in L^{p'}(\real; X^*)$ be a nice function. We then
 have  
 \begin{multline*}
   \left|\int_{\real\times\O} f g\right|
   \le C\int_{\real\times\O} f^\sharp M(g) \\ \le
   C\big\|f^\sharp\big\|_{L^p(\real;X)}\,
   \big\|M(g)\big\|_{L^{p'}(\real;X^*)} \\ \le
   C\big\|f^\sharp\big\|_{L^p(\real;X)}\,
   \big\|g\big\|_{L^{p'}(\real;X^*)}\,,
\end{multline*}
where we have used again
 Bourgain's maximal inequality for $g$ (noting that $X^*$ is also a
 UMD lattice).  Therefore, taking supremum over all $g$ in the unit
 ball of $L^{p'}(\real; X^*)$, we deduce the desired inequality, so
 prove the theorem.

 {\added Finally, observe that the proof above operates with
   individual functions.  This, coupled with the UMD property of~$X$,
   implies that~$X$ can always be assumed separable and it can always
   be equipped with a weak unit.  }
\end{proof}

\section{LPR property for general Banach spaces}
\label{sec:other-than-lattices}

Let~$X$ be a Banach space (not necessarily a lattice).  We shall prove
the following theorem.

\begin{thm}\label{rubioqp}
  If $X$ has the LPR$_q$ for some $2\le q<\8$, then $X$ has the
  LPR$_p$ for any $q \leq p < \8$.
\end{thm}

The proof of the theorem requires some lemmas.

\begin{lem}\label{LPRqsub}
  Assume that $X$ has the $\hbox{LPR}_q$ property.  Let $\left( I_j
  \right)_{j \geq 1}$ be a finite sequence of mutually disjoint
  intervals of $\real$ and $\left( I_{j, k} \right)_{k = 1}^{n_j}$ be
  a finite family of mutually disjoint subintervals of $I_j$ for each
  $j \geq 1$.  Assume that the relative position of $I_{j, k}$ in
  $I_j$ is independent of $j$, i.e., $I_{j,k}-a_j=I_{j',k}-a'_j$
  whenever both $I_{j,k}$ and $I_{j',k}$ are present {\added (i.e., $k
    \leq \min \left\{n_j, n_{j'}\right\}$)}, where $a_j$ is the left
  endpoint of $I_j$.  Then
 \[
 \left\| \sum_{j = 1}^\infty \sum_{k=1}^{n_j}
   \e_j\e_k'S_{I_{j,k}}f\right\|_{L^q(\real;\Rad_2(X))} \le
 c\,\big\|f\big\|_{L^q(\real; X)}\,, \quad\forall\; f\in L^q(\real;X).
 \]
 \end{lem}

\begin{proof}
  We first assume that $\bigcup_{k = 1}^{n_j} I_{j,k}=I_j$ for each $j
  \geq 1$. Note that
  \[S_{I_{j,k}}f=\exp(2\pi\i  a_j\,\cdot)S_{I_{j,k}-a_j}(\exp(-2\pi\i 
  a_j\,\cdot)f).\] Thus, {\added by the contraction principle},
  \[
  \left\| \sum_{j = 1}^\infty \sum_{k = 1}^{n_j} \e_j\e_k'S_{I_{j,k}}f
  \right\|_q \sim  \left\| \sum_{k = 1}^\infty \e_k'\sum_{j:\ n_j \geq
    k}\e_jS_{I_{j,k}-a_j}(\exp(-2\pi\i a_j\,\cdot)f) \right\|_q.
  \]
  Since $X$ has the $\hbox{LPR}_q$ property, so does $\Rad(X)$.
  {\added Let us apply this property of $\Rad(X)$ to the intervals
    $\left( \tilde I_k \right)_{k \geq 1}$ where~$\tilde I_k =
    I_{j,k}-a_j$, for some~$j$ such that~$n_j \geq k$ (for any
    such~$j$ the interval~$I_{j, k} - a_j$ is independent of~$j$ by
    the assumptions of the lemma).  We apply this property to the
    function \[\sum_{k = 1}^\infty \sum_{j:\ n_j \geq k}
    \e_jS_{I_{j,k}-a_j}(\exp(-2\pi\i a_j\,\cdot)f) = \sum_{k =
      1}^\infty S_{\tilde I_k} \left[ \sum_{j:\ n_j \geq k} \epsilon_j
      \left( \exp(-2 \pi \i a_j \cdot) f \right) \right].\]}
  We obtain
  \begin{multline}
    \label{RemarkIIi-op}
    \left\| \sum_{k = 1}^\infty \e_k'\sum_{j:\ n_j \geq k} \e_j
      S_{I_{j,k}-a_j}(\exp(- 2\pi\i a_j\,\cdot)f)\right\|_q \\ \leq c
    \, \left\| \sum_{k = 1}^\infty \sum_{j:\ n_j \geq k}
      \e_jS_{I_{j,k}-a_j}(\exp(-2\pi\i a_j\,\cdot)f) \right\|_q \sim c
    \left\| \sum_{j = 1}^\infty \sum_{k = 1}^{n_j} \e_j
      S_{I_{j,k}}f\right\|_q \\ = c \, \left\| \sum_{j = 1}^\infty
      \e_jS_{I_{j}}f\right\|_q \le c \, \|f\|_{q}.
  \end{multline}
  {\added Assume now that $\bigcup_{k = 1}^{n_j} I_{j,k}\neq I_j$ for
    some $j$.  In this case, consider the family of intervals~$\left(
      \tilde I_k \right)_{k = 1}^\infty$ introduced above.  Observe
    that every~$\tilde I_k \subseteq \left[ 0, +\infty \right)$.
    Observe also that the the right ends of the intervals~$\left( I_j
      - a_j \right)_{j \geq 1}$, that is the points~$b_j - a_j$ do not
    belong to the union~$\cup_{k = 1}^\infty \tilde I_k$.  Let~$\left(
      \tilde I_\el \right)_{\el = 1}^\infty$ be the family of disjoint
    intervals such that $$ \bigcup_{\el = 1}^\infty \tilde I_\el =
    \left[ 0, +\infty \right) \setminus \bigcup_{k = 1}^\infty \tilde
    I_k $$ and such that neither of the points~$\left( b_j - a_j
    \right)_{j = 1}^\infty$ is inner for some~$\tilde I_\el$.  Let
    also~$m_j$ be the maximum number such that the intervals~$\tilde
    I_\el$ with~$\el \leq m_j$ are all to the left of the point~$b_j -
    a_j$.  Set~$I_{j, \el} = \tilde I_\el + a_j$.  Then, $$ I_j =
    \bigcup_{k = 1}^{n_j} I_{j, k} + \bigcup_{\el = 1}^{m_j} I_{j,
      \el}. $$}It is clear that the relative position of
  $(I_{j,k})_{k = 1}^{n_j} \cup (I_{j,\el})_{\el = 1}^{m_j}$ in $I_j$
  is again independent of $j$. 

  {\added Before we proceed, let us re-index the intervals~$\left(
      I_{j, k} \right)_{k = 1}^{n_j}$ and~$\left( I_{j, \el}
    \right)_{\el = 1}^{m_j}$ into a family~$\left( I_{j, s} \right)_{s
      = 1}^{m_j + n_j}$ as follows.  We arrange these intervals from
    left to right within~$I_j$ and index them sequentially from~$1$
    up to~$n_j + m_j$.  Moreover, let~$K_j \subseteq [1, n_j + m_j]$ be
    the subset corresponding to the first family of intervals and~$L_j
    \subseteq [1, n_j + m_j]$ be the subset of indices corresponding
    to the second family of intervals.  Observe that, if~$K = \cup_{j
      = 1}^\infty K_j$ and~$L = \cup_{j= 1}^\infty L_j$, then, for
    every to~$j$, $K_j = K \cap [1, n_j + m_j]$ and, similarly, $L_j =
    L \cap [1, n_j + m_j]$.} Thus by the previous part we get
  {\added $$ \left\| \sum_{j = 1}^\infty \sum_{s = 1}^{n_j + m_j}
      \epsilon_j \epsilon'_s S_{I_{j, s}} f \right\|_{q} \leq c_q\,
    \left\| f \right\|_q.  $$ Observe also that 
    \begin{multline*}
      \sum_{j = 1}^\infty \sum_{s = 1}^{n_j + m_j} \epsilon_j
      \epsilon'_s S_{I_{j, s}} f = \sum_{s = 1}^\infty \sum_{j:\ n_j+
        m_j \geq s} \epsilon_j \epsilon'_s S_{I_{j, s}} f \\ = \sum_{s
        \in K} \sum_{j:\ n_j+ m_j \geq s} \epsilon_j \epsilon'_s
      S_{I_{j, s}} f + \sum_{s \in L } \sum_{j:\ n_j+ m_j \geq s}
      \epsilon_j \epsilon'_s S_{I_{j, s}} f \end{multline*} Thus, by
    taking projection onto the subspace spanned
    by~$\left\{\epsilon'_s\right\}_{s \in K}$, we continue $$ \left\|
      \sum_{s \in K} \sum_{j:\ n_j+ m_j \geq s} \epsilon_j \epsilon'_s
      S_{I_{j, s}} f \right\|_{q} \leq c_q\, \left\| f\right\|_q. $$
    Finally, we observe that $$ \sum_{s \in K} \sum_{j:\ n_j+ m_j \geq
      s} \epsilon_j \epsilon'_s S_{I_{j, s}} f = \sum_{j = 1}^\infty
    \sum_{k = 1}^{n_j} \epsilon_j \epsilon'_k S_{I_{j, k}} f. $$ }
  Hence the lemma is proved.
\end{proof}

The following lemma is interesting in its own right.  We shall only
need its first part.

\begin{lem}\label{cotype2-riesz}
 Let $Y$ be a Banach space. Let $(\Sigma, \nu)$ be a measure space and  $(h_j)\subset L^2(\Sigma)$ a finite sequence.
 \begin{enumerate}[\rm i)]
 \item If $Y$ is of cotype 2 and  there exists a constant $c$ such that
 \[
 \big\|\sum_j\a_jh_j\big\|_2\le c\,\big(\sum_j|\a_j|^2\big)^{1/2}\,,\quad
 \forall\; \a_j\in\com,
 \]
then
  \[
 \big\|\sum_jh_ja_j\big\|_{L^2(\Sigma;Y)} \le c'\,\big\|\sum \e_ja_j\big\|_{\Rad(Y)}
 \,,\quad\forall\; a_j\in Y.
 \]
 \item If $Y$ is of type 2 and there exists a constant $c$ such that
 \[
 \big(\sum_j|\a_j|^2\big)^{1/2}\le c\,\big\|\sum_j\a_jh_j\big\|_2\,,\quad
 \forall\; \a_j\in\com,
 \]
then
  \[
 \big\|\sum \e_ja_j\big\|_{\Rad(Y)}\le c'\,\big\|\sum_jh_ja_j\big\|_{L^2(\Sigma;Y)} 
 \,,\quad\forall\; a_j\in Y.
 \]
 \end{enumerate}
 \end{lem}

\begin{proof}
  i) Let $(a_j)\subset Y$ be a finite sequence. Consider the operator
  $u:\el^2\to Y$ defined by
 $$u(\a)=\sum_j\a_j a_j,\quad \forall\; \a=(\a_j)\in\el^2.$$
 It is well known (see~\cite[Lemma~3.8 and
 Theorem~3.9]{Pisier-Factorization}) that
 \[\pi_2(u)\le c_0 \big\|\sum \e_ja_j\big\|_{\Rad(Y)},\]
 where $c_0$ is a constant depending only on the cotype 2 constant of
 $Y$.  Let $h(\s)=(h_j(\s))_j$ for $\s\in\Sigma$. Then by the
 assumption on $(h_j)$ we get 
 \begin{multline*}
   \big \|\sum_jh_ja_j\big\|_{L^2(\Sigma;Y)}
   = \\ \pi_2(u)\sup\big\{\big(\int_\Sigma|\sum_j\xi_jh_j(s)|^2ds\big)^{1/2}\;:\;
   \xi\in\el^2, \|\xi\|_2\le1\big\}\\
   \le c' \big\|\sum \e_ja_j\big\|_{\Rad(Y)}.
 \end{multline*}

 ii) Let $H$ be the linear span of $(h_j)$ in $L^2(\Sigma)$. Let
 $h_j^*$ be the functional on $H$ such that
 $h_j^*(h_k)=\delta_{j,k}$. We extend $h_j^*$ to the whole
 $L^2(\Sigma)$ by setting $h_j^*=0$ on $H^\perp$. Then $h_j^*\in
 L^2(\Sigma)$ and the assumption implies that
 \[
 \big\|\sum_j\b_jh_j^*\big\|_2\le c
 \big(\sum_j|\b_j|^2\big)^{1/2}\,,\quad\forall\;\b_j\in\com.
 \]
 Now let $(a_j^*)\subset Y^*$ be a finite sequence. Applying i) to
 $Y^*$ and $(h_j^*)$ we obtain \be \big|\sum_j \la a_j^*,\;a_j\ra\big|
 &=&\big|\la \sum_jh_j^*a_j^*,\;\sum_jh_ja_j\ra\big|\\
 &\le& \big\|\sum_jh_j^*a_j^*\big\|_{L^2(\Sigma;Y^*)}\,
 \big\|\sum_jh_ja_j\big\|_{L^2(\Sigma;Y)}\\
 &\le& c' \big\|\sum_j\e_ja_j^*\big\|_{\Rad(Y^*)}\,
 \big\|\sum_jh_ja_j\big\|_{L^2(\Sigma;Y)}\,.  \ee Taking the supremum
 over $(a_j^*)\subset Y^*$ such that $\big\|
 \sum\e_ja_j^*\big\|_{\Rad(Y^*)}\le 1$, we get the assertion.
\end{proof}

Now we proceed to the proof of Theorem \ref{rubioqp}. It is divided into several steps.

\paragraph{The singular integral operator~$T$.}

Let $(I_j)_j$ be a family of disjoint finite intervals and $\psi$ be a
Schwartz function as in Sections~\ref{sec:dyadic-decomposition}
and~\ref{sec:LPR-lattices}. We keep the notation introduced there. We
now set up an appropriate singular integral operator corresponding to
\eqref{a}. It suffices to consider the family $(I_{j,k}^a)_{j,k}$,
$(I_{j,k}^b)_{j,k}$ being treated similarly. Henceforth, we shall
denote $I_{j,k}^a$ simply by $I_{j,k}$. Let $c_{j,k}=a_{j,k} +
2^{k-1}$ for $1\le k\le n_j$. Note that $c_{j,k}$ is the centre of
$I_{j,k}$ if $k<n_j$ and of $\tilde I_{j,k}$ if $k=n_j$.  Define
 \[
 \psi_{j,k}(x)=2^k\exp(2\pi\i c_{j,k}\,x)\,\psi(2^kx)
 \]
 so that the Fourier transform of $\psi_{j,k}$ is adapted to
 $I_{j,k}$, i.e.  \beq\label{fourier (jk)} \chi_{I_{j,k}}\leq \wh
 \psi_{j,k}\leq \chi_{2I_{j,k}}\ \mbox{for}\;
 k<n_j\quad\mbox{and}\quad \chi_{\tilde I_{j,n_j}}\leq \wh
 \psi_{j,n_j}\leq \chi_{2\tilde I_{j,n_j}}\,.  \eeq We should
 emphasise that our choice of $c_{j,k}$ is different from that of
 Rubio de Francia (in~\cite{R1985}) which is $c_{j,k}=n_{j,k}\,2^k$
 for some integer $n_{j,k}$. Rubio de Francia's choice makes his
 calculations easier than ours in the scalar-valued case. The sole
 reason for our choice of $c_{j,k}$ is that $c_{j,k}$ splits into a
 sum of two terms depending on $j$ and $k$ separately. Namely,
 $c_{j,k}=a_j -2+ 2^k+2^{k-1}$.  By \eqref{fourier (jk)},
 \[
 S_{I_{j,k}}f=S_{I_{j,k}}\psi_{j,k}*f.
 \]
 We then deduce, {\added by the splitting property and
   Remark~\ref{RboundedRem},}
 \[
 \big\|\sum_{j,k} \e_j\e'_kS_{I_{j,k}}f\big\|_{p}
 \le c_p\,\big\|\sum_{j,k} \e_j\e'_k\psi_{j,k}*f\big\|_{p}\,.
 \]
 Now write \be \psi_{j,k}*f(x)
 &=&\int 2^k\psi(2^k(x-y))\exp(2\pi\i c_{j,k}(x-y))f(y)dy\\
 &=&\exp(2\pi\i c_{j,k}\,x)
 \int 2^k\psi(2^k(x-y))\exp(-2\pi\i c_{j,k}\,y)f(y)dy\\
 &=& \exp(2\pi\i c_{j,k}\,x)\int K_{j,k}(x,\,y)f(y)dy, \ee where
 \beq\label{K} K_{j,k}(x,\,y)= 2^k\psi(2^k(x-y))\exp(-2\pi\i
 c_{j,k}\,y).  \eeq Using the splitting property of the $c_{j,k}$
 mentioned previously {\added and the contraction principle}, for
 every $x\in\real$ we have \be &&\big\|\sum_{j,k}
 \e_j\e'_k\psi_{j,k}*f(x)
 \big\|_{\Rad_2(X)}\\
 &&~~ = \big\|\sum_{j,k} \e_j\e'_k\exp(2\pi\i c_{j,k}\,x)
 \int K_{j,k}(x,\,y)f(y)dy \big\|_{\Rad_2(X)}\\
 &&~~ \sim \big\|\sum_{j,k} \e_j\e'_k \int K_{j,k}(x,\,y)f(y)dy
 \big\|_{\Rad_2(X)}\,.  \ee Thus we are led to introducing the
 vector-valued kernel $K$: \beq\label{K1} K(x,\,y) =\sum_{j,k}
 \e_j\e'_kK_{j,k}(x,\,y)\in L^2(\O),\quad x,\, y\in\real.  \eeq $K$ is
 also viewed as a kernel taking values in $B(X, \Rad_2(X))$ by
 multiplication. Let $T$ be the associated singular integral operator:
 \[
 T(f)(x)=\int K(x,\,y)f(y)dy,\quad f\in L^p(\real; X).
 \]
By the discussion above, inequality \eqref{a} is reduced to the
boundedness of $T$ from $L^p(\real; X)$ to $L^p(\real; \Rad_2(X))$:
 \beq\label{T1}
 \big\|T(f)\big\|_p\le c_p\, \big\|f\big\|_p\,,\quad
 \forall\;f\in L^p(\real; X).
 \eeq

\paragraph{The~$L^q$ boundedness of~$T$.}

We have the following.

\begin{lem}\label{Lq}
 $T$ is bounded
from $L^q(\real; X)$ to $L^q(\real; \Rad_2(X))$.
 \end{lem}

\pf Let $f\in L^q(\real; X)$. By the previous discussion we have
 \[
 \|Tf\|_q
 \sim\big\|\sum_{j,k}\e_j\e'_k\psi_{j,k}*f\big\|_q.
 \]
By \eqref{fourier (jk)}
 \[
 \sum_{j,k}\e_j\e'_k\psi_{j,k}*f=
 \sum_{j,k}\e_j\e'_k\psi_{j,k}*(S_{2I_{j,k}}f).
 \]
Note that for each $j$ the last interval $I_{j,n_j}$ above should be the dyadic interval $\wt I_{j,n_j}$.   We claim that
 \[
 \big\|\sum_{j,k}\e_j\e'_k\psi_{j,k}*g_{j,k}\big\|_q
 \le c\big\|\sum_{j,k}\e_j\e'_k g_{j,k}\big\|_q,\quad
 \forall\; g_{j,k}\in L^q(\real; X).
 \]
Indeed, using the splitting property of the $c_{j,k}$ we have
 \[
 \big\|\sum_{j,k}\e_j\e'_k\psi_{j,k}*g_{j,k}\big\|_q
 \sim \big\|\sum_{j,k}\e_j\e'_k\wt\psi_{j,k}*\wt g_{j,k}\big\|_q,
 \]
where
 \[
 \wt\psi_{j,k}(x)=2^k\psi(2^kx)\quad\mbox{and}\quad
 \wt g_{j,k}(x)=\exp(-2\pi\i c_{j,k}\,x)g_{j,k}(x).
 \]
For $x\in\real$ define the operator $N(x): \Rad_2(X)\to  \Rad_2(X)$ by
 \[
 N(x)\big(\sum_{j,k}\e_j\e'_ka_{j,k}\big)=\sum_{j,k}\e_j\e'_k\wt\psi_{j,k}(x)a_{j,k}.
 \]
It is obvious that $N: \real\to B(\Rad_2(X))$ is a smooth function and
 \[
  \sum_{j,k}\e_j\e'_k\wt\psi_{j,k}*\wt g_{j,k}=N*\wt g \quad\mbox{with}\quad
  \wt g=\sum_{j,k}\e_j\e'_k\wt g_{j,k}.
  \]
  It is also easy to check that $N$ satisfies
  \cite[Theorem~3.4]{weis}.  Since $\Rad_2(X)$ is a UMD space, it
  follows from~\cite{weis} that the convolution operator with $N$ is
  bounded on $L^q(\real; \Rad_2(X))$.  Thus
 \[
 \big\|\sum_{j,k}\e_j\e'_k\wt\psi_{j,k}*\wt g_{j,k}\big\|_q
 \le c \big\|\sum_{j,k}\e_j\e'_k\wt g_{j,k}\big\|_q.
 \]
Using again the splitting property of the $c_{j,k}$ and going back to the $g_{j,k}$, we prove the claim. Consequently, we have
 \[
 \|T (f) \|_q\le c \big\|\sum_{j,k}\e_j\e'_kS_{2I_{j,k}}f\big\|_q.
 \]
We split the family $\big\{2I_{j,k}\big\}$ into three subfamilies $\big\{2I_{j,3k+\el}\big\}$ of disjoint intervals with $\el\in\{0, 1, 2\}$.
Accordingly, we have
 \[
 \|T (f) \|_q\le c \sum_{\el=0}^2\big\|\sum_{j,k}\e_j\e'_kS_{2I_{j,3k+\el}}f\big\|_q.
 \]
Each subfamily $\big\{2I_{j,3k+\el}\big\}_{j,k}$ satisfies the condition of Lemma~\ref{LPRqsub}. Hence
 \[
 \big\|\sum_{j,k}\e_j\e'_kS_{2I_{j,3k+\el}}f\big\|_q\le c \|f\|_q.
 \]
 Thus the lemma is proved.\cqd

 \paragraph{An estimate on the kernel $K$.} This subsection contains the
 key estimate on the kernel $K$ defined in \eqref{K1}. Fix
 $x,\,z\in\real$ and an integer $m\geq 1$. Let
 \[
 I_m(x,z)=\big\{y\in\real\;:\; 2^m|x-z|<|y-z|
 \leq 2^{m+1}|x-z|\big\}.
 \]

\begin{lem}\label{kernel estimate}
  If~$X^*$ is of cotype~$2$ and if~$(\l_{j,k})\subset X^*$, then \be
  \int_{I_m(x,z)}\big\|
  \sum_{j,k}[K_{j,k}(x,\,y)-K_{j,k}(z,\,y)]\l_{j,k} \big\|_{X^*}^2dy
  \le
  c\,\frac{\big\|\sum_{j,k}\e_j\e'_k\l_{j,k}\big\|^2_{\Rad_2(X^*)}}
  {2^{5m/3}|x-z|}\,.  \ee
\end{lem}

\pf Let $(\l_{j,k})\subset X^*$ such that
 \[\big\|\sum_{j,k}\e_j\e'_k\l_{j,k}\big\|_{\Rad_2(X^*)}\le1.\]
By the definition
of $K_{j,k}$ in \eqref{K}, we have
 \[
 \sum_{j,k}[K_{j,k}(x,\,y)-K_{j,k}(z,\,y)]\l_{j,k}
 =\sum_{k}\mu_k2^{k}\,[\psi(2^k(x-y))-\psi(2^k(z-y))]\,q_k(y)\,,
 \]
where
 \[
 \mu_k=\big\|\sum_{j}\e_j\l_{j,k}\big\|_{\Rad(X^*)}\quad\textrm{and}\quad
 q_k(y)=\mu_k^{-1}\sum_{j}\l_{j, k}\exp(-2\pi\i c_{j,k}\,y).
 \]
Since $\Rad(X^*)$ is of cotype 2,
 \[
 \sum_k\mu_k^2\le c
 \big\|\sum_k\e'_k \sum_{j}\e_j\l_{j,k}\big\|^2_{\Rad(\Rad(X^*))}
 \le c.
 \]
 Thus \be &&\int_{I_m(x,z)}\big\|
 \sum_{j,k}[K_{j,k}(x,\,y)-K_{j,k}(z,\,y)]\l_{j,k}
 \big\|_{X^*}^2dy\\
 &&~~ \le \sum_{k}2^{2k}\sup_{y\in
   I_m(x,z)}|\psi(2^k(x-y))-\psi(2^k(z-y))|^2\int_{I_m(x,z)}\|q_k(y)\|_{X^*}^2dy.
 \ee Note that for fixed $k$ \beq\label{gap} \left|
   c_{j,k}-c_{j^{\prime},k} \right|\geq 2^k,\quad \forall\; j\neq
 j^{\prime}\,.  \eeq Now we appeal to the following classical
 inequality on Dirichlet series with small gaps. Let $(\g_j)$ be a
 finite sequence of real numbers such that
 \[
 \g_{j+1}-\g_j\ge 1,\quad\forall\;j\ge1.
 \]
 Then, by~\cite[Ch.~V, Theorem~9.9]{Zydmund-TS}, for any interval
 $I\subset\real$ and any sequence $(\a_j)\subset\com$
 \[
 \int_I\big|\sum_j\a_j\exp(2\pi\i \g_j\,y)\big|^2dy
 \le c\max(|I|, \;1)\sum_j|\a_j|^2\,,
 \]
 where $c$ is an absolute constant. Applying this to the function
 $q_k$, using Lemma \ref{cotype2-riesz} and \eqref{gap}, we find \be
 \int_{I_m(x,z)}\|q_k\|_{X^*}^2dy &\le&
 c\,2^{-k}\max(2^k|I_m(x,z)|,\;1)\,
 \mu_k^{-2}\big\|\sum_{j}\e_j\l_{j,k}\big\|^2_{\Rad(X^*)}\\
 &=& c\,\max(2^m|x-z|, \;2^{-k})\,.  \ee Let
 \begin{multline*}
   k_0=\min\big\{k\in\nat:\; 2^{-k}\leq 2^m|x-z|\big\} \ \
   \text{and}\\ k_1=\min\big\{k\in\nat:\; 2^{-k}\leq
   2^{2m/3}|x-z|\big\}.
 \end{multline*}
 Note that $k_0\leq k_1$. For $k\leq k_1$ we have \be
 |\psi(2^k(x-y))-\psi(2^k(z-y))|\leq c\,2^k|x-z|.  \ee {\added Recall
   that~$\psi$ is a Schwartz function, in particular~$\left|
     x\right|^2 \left| \psi(x) \right| \leq c$.  Thus, for $k\geq
   k_1$, we have} \be |\psi(2^k(x-y))-\psi(2^k(z-y))|\leq
 c\,2^{-2k}|y-z|^{-2} \leq c\,2^{-2k-2m}|x-z|^{-2}\,, \ee {\added
   where the second estimate comes from the fact that~$y \in I_{m}(x,
   z)$.}  Let
\[
\alpha_k=2^{2k}\sup_{y\in
I_m(x,z)}|\psi(2^k(x-y))-\psi(2^k(z-y))|^2\int_{I_m(x,z)}\|q_k(y)\|_X^2dy.
\]
Combining the preceding inequalities, we
deduce the following estimates on $\a_k$:
 \be
 \a_k
 &\le&
 c\,2^{2k}2^{2k}|x-z|^22^{-k}=c\,2^{3k}|x-z|^2
 \quad\mbox{for}\quad k\le k_0;\\
 \a_k
 &\le& c\,2^{2k}2^{2k}|x-z|^22^m|x-z|=c\,2^{4k}2^m|x-z|^3
 \quad\mbox{for}\quad k_0< k< k_1;\\
 \a_k
 &\le&
 c\,2^{2k}(2^{k+m}|x-z|)^{-4}2^m|x-z|=c\,2^{-2k}2^{-3m}|x-z|^{-3}
 \quad\mbox{for}\quad k\ge k_1.
 \ee
Therefore,
 \be
 &&\int_{I_m(x,z)}\big\|
 \sum_{j,k}[K_{j,k}(x,\,y)-K_{j,k}(z,\,y)]\l_{j,k}
 \big\|_{X^*}^2dy\\
 &&~~\le\sum_{1\leq k\leq k_0}\a_k
 +\sum_{k_0< k<k_1}\a_k+\sum_{k\geq k_1}\a_k \\
 &&~~\le
 c\,\big[2^{3k_0}|x-z|^2+2^{4k_1}2^m|x-z|^3+2^{-2k_1}2^{-3m}|x-z|^{-3}\big]\\
 &&~~ \le c\,2^{-5m/3}|x-z|^{-1}\, .
 \ee
This is the desired estimate for the kernel $K$. \cqd

\paragraph{The $L^\8$-{\rm BMO} boundedness.} Recall that $T$ is the
singular integral operator associated with the kernel $K$.

\begin{lem}\label{Linfty-BMO}
 The operator $T$ is
bounded from $L^\infty(\real;X)$ to ${\rm
BMO}(\real;\Rad_2(X))$.
\end{lem}

\n {\bf Proof}.  {\added Recall that $$ \left\| g \right\|_{\rm
    BMO(\real; X)} \leq 2\, \sup_{I \subseteq \real} \frac 1{\left| I
    \right|} \int_I \left\| g(x) - b_I \right\|_X\, dx, $$
  where~$\left\{b_I\right\}_{I \subseteq \real} \subseteq X$ is any
  family of elements of~$X$ assigned to each interval~$I \subseteq
  \real$.}  Fix a function $f\in L^\infty(\real; X)$ with
$\|f\|_{\8}\le 1$ and an interval $I\subset \real$. Let $z$ be the
centre of $I$ and let
 \[
 b_I=\int_{(2I)^c}K(z,y)f(y)\,dy.
 \]
 Then{\added , for~$x \in I$,}
 \[
 Tf(x)-b_I
 =\int_{(2I)^c}[K(x,y)-K(z,y)]f(y)\,dy+ \int_{2I}K(x,y)f(y)dy.
 \]
Thus
 \be
 &&\frac1{|I|}\int_I \big\|Tf(x)-b_I\big\|_{\Rad_2(X)}dx\\
 &&~~ \le\frac1{|I|}\int_I \big\|\int_{(2I)^c}[K(x,y)-K(z,y)]f(y)\,dy\big\|_{\Rad_2(X)}dx\\
 &&~~+\frac1{|I|}\int_I \big\|\int_{2I}K(x,y)f(y)dy\big\|_{\Rad_2(X)}dx\\
 &&~~ {\mathop =^{\rm def}}\;A+B.
  \ee
  By Lemma~\ref{Lq} we have
 \[
 B\le |I|^{-1/q}\big\|T(f\chi_{2I})\big\|_q\le c.
 \]
To estimate $A$, fix $x\in I$. Choose $(\l_{j,k})\subset X^*$ such that
 \[
 \big\|\sum_{j,k}\e_j\e'_k\l_{j,k}\big\|_{\Rad_2(X^*)}\le1.
 \]
and
\begin{multline*}
  \big\|\int_{(2I)^c}[K(x,y)-K(z,y)]f(y)\,dy\big\|_{\Rad_2(X)} \\ \sim
  \sum_{j,k}\la\l_{j,k},\;
  \int_{(2I)^c}[K_{j,k}(x,y)-K_{j,k}(z,y)]f(y)\,dy\ra
\end{multline*}
Then by  Lemma \ref{kernel estimate}, we find
 \be
 &&\big\|\int_{(2I)^c}[K(x,y)-K(z,y)]f(y)\,dy\big\|_{\Rad_2(X)}\\
 &&~~ \le \int_{(2I)^c}
 \big\|\sum_{j,k}[K_{j,k}(x,y)-K_{j,k}(z,y)]\l_{j,k}\big\|_{X^*}dy\\
 &&~~ \le  \sum_{m=1}^\infty|I_m(x,z)|^{1/2}\big(\int_{I_m(x,z)}
 \big\|\sum_{j,k}[K_{j,k}(x,y)-K_{j,k}(z,y)]\l_{j,k}\big\|_{X^*}^2dy
 \big)^{1/2}\\
 &&~~\le c\sum_{m=1}^\infty(2^m|x-z|)^{1/2}(2^{5m/3}|x-z|)^{-1/2}\\
 &&~~ \le  \sum_{m=1}^\infty c\,2^{-m/3}\leq c.
 \ee
Therefore, $A\le c$. Thus $T$ is bounded from $L^\infty(\real;X)$ to ${\rm
BMO}(\real;\Rad_2(X))$. \cqd

Combining the result of Lemma~\ref{Linfty-BMO} and Lemma~\ref{Lq} and
applying interpolation (see~\cite{blasco-xu-1989}), we immediately see
that the operator~$T$ is bounded from $L^p(\real; X)$
to $L^p(\real; \Rad_2 (X))$ for every $q<p<\8$. Thus  Theorem~\ref{rubioqp} is proved.

\begin{rk}
  Let $$ T(f)^\sharp(x)=\sup_{x\in I}\frac1{|I|}\, \int_I
  \big\|T(f)(y)- T(f)_I\big\|_{\Rad_2(X)}\,dy $$ and $$
  M_q(f)(x)=\sup_{x\in I} \left( \frac1{|I|}\, \int_I
    \big\|f(y)\big\|^q_{X}\,dy \right)^{\frac 1q}. $$ Under the
  assumption of Theorem~\ref{rubioqp} one can show the following
  pointwise estimate: \[T(f)^\sharp\le c\, M_q(f). \]
\end{rk}


\providecommand{\bysame}{\leavevmode\hbox to3em{\hrulefill}\thinspace}

\end{document}